\documentclass[
	paper=A4,
	pagesize=auto,
	fontsize=12pt
]{scrartcl}
\usepackage{libertine}
\recalctypearea

\usepackage[british]{babel}
\usepackage[style=alphabetic,url=false,isbn=false,doi=true,eprint=false,backend=biber]{biblatex}
\AtEveryBibitem{\clearlist{language}}
\usepackage{graphicx}
\usepackage{amsmath,amssymb,amsthm}
\usepackage{csquotes}
\usepackage{bm}
\usepackage{leftidx}

\usepackage{enumitem}
\setlist[enumerate]{label = (\roman*)}

\usepackage{zref-clever}
\usepackage{hyperref}
\usepackage{xurl}
\hypersetup{colorlinks,breaklinks=true}

\AddToHook{env/proposition/begin}{%
\zcsetup{countertype={theorem=proposition}}}
\AddToHook{env/lemma/begin}{%
\zcsetup{countertype={theorem=lemma}}}
\AddToHook{env/corollary/begin}{%
\zcsetup{countertype={theorem=corollary}}}

\theoremstyle{plain}
\newtheorem{theorem}{Theorem}[section]
\newtheorem{proposition}[theorem]{Proposition}
\newtheorem{lemma}[theorem]{Lemma}
\newtheorem{corollary}[theorem]{Corollary}

\AddToHook{env/example/begin}{%
\zcsetup{countertype={theorem=example}}}
\AddToHook{env/remark/begin}{%
\zcsetup{countertype={theorem=remark}}}

\theoremstyle{remark}
\newtheorem{example}[theorem]{Example}
\newtheorem{remark}[theorem]{Remark}

\AddToHook{env/definition/begin}{%
\zcsetup{countertype={theorem=definition}}}

\theoremstyle{definition}
\newtheorem{definition}[theorem]{Definition}

\newlength{\JZHeightOfX}
\newcommand{\JZOrcidlink}[1]{
\setlength{\JZHeightOfX}{\fontcharht\font`X}
\includegraphics[height=\JZHeightOfX]{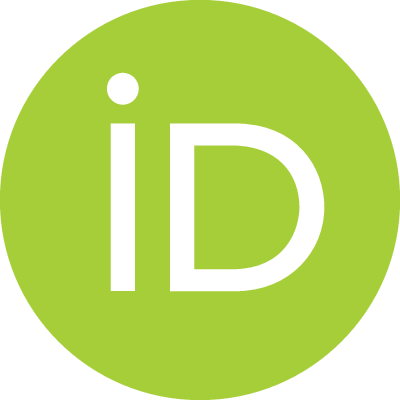}
\href{https://orcid.org/#1}{#1}
}

\begin{document}
\title{Countability Properties of Weakly Compact Sets in Asymmetric Locally Convex Spaces}
\author{
Jobst Ziebell\hspace{2em}\JZOrcidlink{0000-0002-9715-6356}\\
\small{Faculty of Mathematics and Computer Science, Friedrich-Schiller-University, Jena, Germany}
}
\date{\today}
\maketitle

\begin{abstract}
A dual pair formulation for asymmetric locally convex spaces is introduced that strictly generalises the ordinary vector space setting.
The concept of a polar topology carries over to the asymmetric case and some familiar results are reproduced or generalised such as the bipolar theorem and a Mackey-Arens type theorem.
The implications of weak compactness and countability properties are studied and appear intimately connected to separation properties.
An asymmetric analogue $C_a(K)$ of the well-known space $C_p(K)$ is introduced and the properties of (relatively) (countably respectively sequentially) compact subspaces are investigated.
In particular, it is shown that $C_a([0,1])$ is not angelic.
However, for Hausdorff subspaces, the different compactness conditions are equivalent and imply the Fréchet-Urysohn property.
Moreover, an analogue of the Eberlein-Šmulian theorem holds in asymmetrically normed spaces.
\end{abstract}

\section{Introduction}
The goal of this paper is to derive an asymmetric analogue of the Eberlein-Šmulian theorem and of particular interest are the sequential properties of weakly compact sets in asymmetrically normed spaces.
A motivation to study these comes from a simple consideration in measure theory:
Suppose one has a finite measure $\mu$ on some measurable space $(X, \mathcal{A})$ and a sequence $(f_n)_{n \in \mathbb{N}}$ of real-valued $\mu$-measurable functions on $X$.
One might consider the convergence properties of the sequence
\begin{equation}
n \mapsto \int_X \exp \left[ f_n \right] \mathrm{d} \mu \, .
\end{equation}
If it converges to a finite number, one immediately knows that for large enough $N \in \mathbb{N}$, the functions $(\max\{ f_n, 0 \})_{n = N}^\infty$ are uniformly $\mu$-integrable.
Using Komlos' theorem and the convexity of the integrand, it follows that $\{ \max\{ \sum_{m = N}^n f_m / n, 0 \} : n \in \mathbb{N}_{\ge N} \}$ is relatively compact in $L^1(\mu)$.
At the same time, one knows very little about the negative parts of the $f_n$s and thus not even whether $f_n \in L^1(\mu)$ for large $n \in \mathbb{N}$.
This phenomenon is easily captured by the asymmetric seminorm $f \mapsto \int_X \max \{f, 0 \} \mathrm{d} \mu$ leading to an asymmetrically normed space.
In fact, the set $\{ f_n : n \in \mathbb{N}_{\ge N} \}$ has powerful compactness properties in the weak topology of this asymmetrically normed space.
If this set in question has sequential properties, these enable conclusions using Fatou's lemma, the dominated convergence theorem and many more.

The natural setting of the Eberlein-Šmulian theorem is the stage of dual pairs.
However, this is a concept that has - to the knowledge of the author - not yet been treated in the context of asymmetric locally convex spaces.
A considerably more general framework has been developed for locally convex cones in \cite[Chapter II.3]{src:KeimelRoth:OrderedConesAndApproximation}, but these are built on rather abstract order structures and lack a direct connection to the comparatively simple concept of a family of seminorms.
Thus, the purpose of this paper is threefold:
\begin{itemize}
\item Develop a simple formulation of asymmetric dual pairs similar to the one in the realm of locally convex vector spaces.
\item Investigate the interplay of countability properties and weak compactness.
\item Apply the results to the weak topology of an asymmetrically normed space.
\end{itemize}
In the symmetric case, it is very instructive to embed a normed space with its weak topology in a function space $C_p(K)$, that is the set of continuous functions on a compact Hausdorff space equipped with the topology of pointwise convergence.
In the asymmetric considerations, the analogous space $C_a(K)$ becomes an asymmetric version of $C_p(K)$.
Unfortunately, the angelicity of $C_p(K)$ does not survive the passage to $C_a(K)$, but very useful compactness and sequentiality properties can be derived nonetheless.
Moreover, the topological behaviour of Hausdorff subsets of $C_a(K)$ turn out be much richer than that of subsets of $C_p(K)$.

In \zcref{sec:Preliminaries}, the basic notions are introduced and some technicalities which may be unfamiliar are given.
\zcref[S]{sec:AsymmetricDualPairs} develops the notion of an asymmetric dual pair as a pairing of a vector space and a cone.
Using this formulation, polar topologies, in particular the weak and weak-$\ast$ topologies are defined and the bipolar theorem as well as a Mackey-Arens type theorem is shown.
It is also proven that every $T_0$ asymmetric locally convex topology is naturally given as a polar topology.
In \zcref{sec:WeakCompactnessAndCountability}, the implications of weak compactness and countability properties are elucidated.
These are applied to the space $C_a(K)$ in \zcref{sec:CaK} where it is shown that some topological properties of $C_p(K)$ carry over to $C_a(K)$.
Moreover, it is shown that $C_a([0,1])$ is not angelic and one can trace the failure of angelicity to the possible non-Hausdorffness of compact subsets.
Finally, the straightforward consequences for the weak topology of an asymmetrically normed space are given in \zcref{sec:WeakTopologyOnAsymmetricallyNormedSpace}.
\section{Preliminaries}
\label{sec:Preliminaries}
\subsection{Some Topological Notions}
A topological space is \textbf{not} implicitly assumed to be Hausdorff or satisfy any separation axiom at all.
It is \textbf{regular}, if every point has a neighbourhood basis consisting of closed sets.
A subset $A$ of a topological space $X$ is \textbf{compact}, if every open cover of $A$ has a finite subcover and following \cite{src:KönigKuhn:AngelicSpaces}, we shall call $A$ \textbf{relatively compact} in $X$ if every open cover of $X$ has a finite subcover of $A$.
Similarly, $A$ is \textbf{countably compact}, if every countable open cover has a finite subcover and $A$ is \textbf{relatively countably compact} in $X$, if every countable open cover of $X$ has a finite subcover of $A$.
Finally, $A$ is \textbf{sequentially compact}, if every sequence in $A$ has a convergent subsequence in $A$ and is \textbf{relatively sequentially compact} in $X$ if every sequence in $A$ has a convergent subsequence in $X$.
For completeness, we state the following:
\begin{corollary}
\label{cor:RelativeCompactnessFromClusterPoints}
$A$ is relatively compact in $X$ if and only if every net $(x_\alpha)_{\alpha \in I}$ in $A$ has a cluster point in $X$.
$A$ is relatively countably compact in $X$ if and only if every sequence $(x_n)_{n \in \mathbb{N}}$ in $A$ has a cluster point in $X$.
\end{corollary}
Note that if $X$ is regular, then $A$ is relatively compact if and only if its closure $\bar{A}$ is compact \cite[Proposition 17.15]{src:Schlechter:HandbookOfAnalysis}.
\begin{definition}[{\cite{src:KönigKuhn:AngelicSpaces}}]
\label{def:angelicity}
A topological space $X$ is \textbf{angelic}, if for every relatively countably compact subset $A \subseteq X$,
\begin{enumerate}
\item\label{def:angelicity:RelativeCompactnessEquivalence} $A$ is relatively compact,
\item for each $a \in \mathrm{cl}\, A$, there is a sequence $(a_n)_{n \in \mathbb{N}}$ in $A$ converging to $a$.
\end{enumerate} 
\end{definition}
A \textbf{network} in a topological space $X$ is a collection $\mathcal{N}$ of subsets of $X$ such that for every $x \in X$ and every neighbourhood $U$ of $x$, there is some $N \in \mathcal{N}$ with $x \in N \subseteq U$.
A topological space $X$ is \textbf{Lindelöf}, if every open cover has a countable subcover of $X$.
It is \textbf{hereditarily Lindelöf} if the same is true for every subspace of $X$.
The following \zcref[noref]{lem:CountableNetworkImpliesHereditaryLindelöf} is well-known, but stated here for completeness.
\begin{lemma}
\label{lem:CountableNetworkImpliesHereditaryLindelöf}
Let $X$ be a topological space with a countable network.
Then $X$ is hereditarily Lindelöf.
\end{lemma}
%
Networks are of particularly elegant use in conjunction with compactness properties.
The following metrisability result is classical.
\begin{theorem}
\label{thm:CompactHausdorffCountableNetworkImpliesMetrisable}
Let $X$ be a compact Hausdorff topological space with countable network.
Then $X$ is metrisable.
Conversely, every compact, metrisable topological space has a countable network consisting of closed sets.
\end{theorem}
We also have a variant for countably compact sets.
\begin{lemma}
\label{lem:USCSeparation}
Let $X$ be a countably compact space and $F$ a countable family of real-valued upper semicontinuous functions on $X$ with the property that
\begin{equation}
\forall x, y \in X \exists f \in F : f \left( x \right) < f \left( y \right) \, .
\end{equation}
Then $X$ is compact, $T_1$ and has a countable network.
\end{lemma}
\begin{proof}
That $X$ is $T_1$ is immediately clear.
For each $f \in F$ and every $q \in \mathbb{Q}$, consider the closed set
\begin{equation}
C \left( f, q \right) = \left\{ x \in X : f \left( x \right) \ge q \right\}
\end{equation}
and let $\mathcal{N}$ be the collection of finite intersection of all such sets.
Then $\mathcal{N}$ is countable and, in fact, a network of $X$.
To see that, let $x \in X$ and $U$ be an open neighbourhood of $x$.
Clearly, $K = X \setminus U$ is countably compact and for every $y \in K$, we find some $f_y \in F$ with $f_y(y) < f_y(x)$.
Thus, there is some $q_y \in \mathbb{Q}$ such that $x \in C ( f_y, q_y )$ and $y \notin C ( f_y, q_y )$ and we may set $V_y = X \setminus C ( f_y, q_y )$.
Since $\mathcal{N}$ is countable, $\{V_y : y \in K \}$ is a countable open cover of $K$.
Hence, there are $y_1, \dots, y_n \in K$ with $\cup_{m = 1}^n V_{y_m} \supseteq K$.
Consequently,
\begin{equation}
x
\in
\bigcap_{m = 1}^n C ( f_{y_m}, q_{y_m} )
=
\bigcap_{m = 1}^n \left[ X \setminus V_{y_m} \right]
=
X \setminus \bigcup_{m = 1}^n V_{y_m}
\subseteq
X \setminus K
=
U \, .
\end{equation}
Thus, $X$ has a countable network and by \zcref{lem:CountableNetworkImpliesHereditaryLindelöf}, $X$ is Lindelöf and thus compact.
\end{proof}
For a topological space $X$, we define the \textbf{specialisation order}
\begin{equation}
\forall x, y \in X: x \le y \iff y \in \mathrm{cl}\, \{ x \} \, .
\end{equation}
Then, another (abstract nonsense) compactness classification is the following.
\begin{proposition}
\label{prop:T0CompactnessAbstractNonsense}
Let $X$ be a countably compact $T_0$ space with specialization order $\le$ and the property that every countably compact $T_1$ subspace is compact.
Then $X$ is compact if and only if every totally ordered subset of $X$ has an upper bound.
\end{proposition}
\begin{proof}
$\Rightarrow$:
Let $T \subseteq X$ be a totally ordered subset.
Then $(t)_{t \in T}$ is a net in $X$ which by compactness, has a cluster point $x \in X$.
Since $T$ is totally ordered, it follows that $t \le x$ for all $t \in T$ and thus $x$ is an upper bound.

$\Leftarrow$:
Applying Zorn's lemma, the set $\mathrm{Max}\, X$ of maximal elements of $X$ is such that
\begin{equation}
\forall x \in X \exists m \in \mathrm{Max}\, X : x \le m \, .
\end{equation}
Hence, every cover of $\mathrm{Max}\, X$ by sets that are open in $X$ is also an open cover of $X$.
Consequently, $\mathrm{Max}\, X$ is countably compact.
Since it is $T_1$ by definition, it is also compact by assumption.
Using the above covering property again, the compactness of $X$ follows.
\end{proof}
A subset $A$ of topological space $X$ is \textbf{locally finite in $X$}, if every point $x \in X$ has a neighbourhood containing only finitely many points of $A$.
Finally, a topological space $(X, \tau)$ is a \textbf{$D$-space}, if for every neighbourhood assignment $U: X \to \tau$, there exists a subset $A \subseteq X$ that is locally finite in $X$ and for which $U(A)$ covers $X$.
In particular, a countably compact space is compact if and only if it is a $D$-space \cite[p. 655]{src:ArhangelskiiBuzyakova:DSpaces}.
Following \cite{src:Tkachuk:MonolithicSpaces}, a family $\mathcal{N}$ of subsets of a topological space $X$ is an \textbf{external network} of a subset $A \subseteq X$ if for any $x \in A$ and open neighbourhood $U$ of $x$, there is some $N \in \mathcal{N}$ with $x \in N \subseteq U$.
\begin{definition}
\label{def:MonotonicMonolithicity}
A topological space $X$ is \textbf{monotonically monolithic} if there is an assignment $\mathcal{N} : \mathcal{P}(X) \to \mathcal{P}(X)$ such that for all subsets $A \subseteq X$, $\mathcal{N}(A)$ is an external network of $\overline{A}$ and
\begin{enumerate}
\item \label{itm:MonotonicMonolithicity:CardinalityBound} $| \mathcal{N}(A) | \le \max\{ | A |, | \mathbb{N} | \}$,
\item \label{itm:MonotonicMonolithicity:Monotonicity} $A \subseteq B \subseteq X \implies \mathcal{N}(A) \subseteq \mathcal{N}(B)$,
\item \label{itm:MonotonicMonolithicity:Continuity} Given an ordinal $\alpha$ and a family $\{ A_\beta : \beta < \alpha \}$ of increasing subsets of $X$, then $\mathcal{N}( \cup_{\beta < \alpha} A_\beta ) = \cup_{\beta < \alpha} \mathcal{N}( A_\beta )$.
\end{enumerate}
\end{definition}
A fundamental result is that monotonically monolithic $T_1$ spaces are $D$-spaces.
This was originally proven in \cite[Theorem 2.14]{src:Tkachuk:MonolithicSpaces} with an implicit assumption that the underlying spaces are Tychonoff.
A more general reference containing the statement for $T_1$ spaces is \cite{src:Peng:WeaklyMonotonicallyMonolithicSpaces}.
\subsection{Asymmetric Locally Convex Spaces}
Vector spaces and cones are always considered over the field of real numbers.
A function $p : X \to [0, \infty)$ on a vector space or a cone $X$ is an \textbf{asymmetric seminorm}, if
\begin{equation}
p \left( r x \right) = r p \left( x \right)
\qquad
p \left( x + y \right) \le p \left( x \right) + p \left( y \right) \, ,
\end{equation}
for all $r \ge 0$ and all $x, y \in X$.
An ordinary seminorm is clearly an asymmetric seminorm, but the converse does not hold because only nonnegative scalars are considered.
In particular, we do not require that $p(-x) = p(x)$.
Also note that this deviates from the asymmetric seminorms on \emph{cones} considered in \cite[Section 2.3.3]{src:Cobzaş:FunctionalAnalysisinAsymmetricNormedSpaces}, which include a form of definiteness as well.
An asymmetric seminorm $p$ on $X$ for which
\begin{equation}
\forall x \in X : p(x) = 0 \text{ and } p(-x) = 0 \implies x = 0 \, ,
\end{equation}
is an \textbf{asymmetric norm}.
Given a vector space or a cone $X$, a family of asymmetric seminorms $\mathcal{P}$ on $X$ induces a topology $\tau_{\mathcal{P}}$ on $X$ with a basis given by sets of the form
\begin{equation}
\label{eq:AsymmetricTopology}
U_{x, F, \epsilon}
=
\left\{ y \in X \middle| \exists \Delta \in X : y = x + \Delta \text{ and } \forall p \in F : p \left( \Delta \right) < \epsilon \right\} \, ,
\end{equation}
for every $x \in X$, $\epsilon > 0$ and every finite subset $F$ of $\mathcal{P}$.

In the remainder of this section, $X$ will always denote a vector space.
If $\mathcal{P}$ consists of ordinary seminorms, the resulting topology is that of a (possibly non-Hausdorff) locally convex topological vector space.
The pair $(X, \tau_{\mathcal{P}})$ is an \textbf{asymmetric locally convex space} and if $\mathcal{P} = \{ p \}$ for an asymmetric norm $p$, it is also written as $(X, p)$ and called an \textbf{asymmetrically normed space}.
\begin{example}
$(\mathbb{R}, p)$ with $p : \mathbb{R} \to \mathbb{R}, x \mapsto \max\{ x, 0 \}$ is an asymmetrically normed space.
The induced topology is the one generated by all open intervals of the form $(- \infty, x )$ for $x \in \mathbb{R}$.
Because all open sets have nonempty intersection, this space is not Hausdorff.
It is easily seen to be $T_0$ and not $T_1$.
\end{example}
Many basic properties of such spaces can be found in \cite{src:Cobzaş:FunctionalAnalysisinAsymmetricNormedSpaces}.
In particular,
\begin{itemize}
\item addition is continuous while scalar multiplication in general is not, \cite[Section1.1.7]{src:Cobzaş:FunctionalAnalysisinAsymmetricNormedSpaces},
\item $\tau_{\mathcal{P}}$ is $T_0$ if and only if, for every $x \in X \setminus \{ 0 \}$, there is some $p \in \mathcal{P}$ with $p(x) > 0$ or $p(-x) > 0$ \cite[Proposition 1.1.63]{src:Cobzaş:FunctionalAnalysisinAsymmetricNormedSpaces},
\item every $p \in \mathcal{P}$ is upper semicontinuous (by definition of $\tau_{\mathcal{P}}$).
\end{itemize}
This should be compared to locally convex spaces, for which
\begin{itemize}
\item scalar multiplication is always continuous,
\item the topology is $T_0$ if and only if it is Hausdorff,
\item every $p \in \mathcal{P}$ is continuous.
\end{itemize}
It is therefore clear, that asymmetric locally convex spaces can exhibit much more general topological behaviour than ordinary locally convex spaces.
Because $p(x-y)$ might differ from $p(y-x)$, it is clearly a convention which one should be considered small in \zcref{eq:AsymmetricTopology}.
To be able to reason about both cases, the \textbf{conjugate} family of asymmetric seminorms is defined as $\overline{\mathcal{P}} = \{ \overline{p} : p \in \mathcal{P} \}$, where $\overline{p}(x) = p(-x)$ for all $x \in X$.
Then, $(X, \overline{\mathcal{P}})$ is the conjugate asymmetric locally convex space and the topology generated by the seminorms $\mathcal{P}_s = \{ p + \overline{p} : p \in \mathcal{P} \}$ is a locally convex topological vector space.
We thus have the \emph{three} topologies $\tau_{\mathcal{P}}$, $\tau_{\overline{\mathcal{P}}}$ and $\tau_{\mathcal{P}_s}$ which are equal if all elements in $\mathcal{P}$ are seminorms but may all be different otherwise.

To an asymmetric locally convex space $X$, we associate the \textbf{dual space} $X^*$ given as the collection of all linear and \emph{upper semicontinuous}, real-valued functions on $X$.
It is easily verified, that this definition coincides with the dual space of a locally convex topological vector space.
Moreover, $X^*$ is clearly a cancellative cone.
\section{Asymmetric Dual Pairs}
\label{sec:AsymmetricDualPairs}
We begin by generalising the notion of a dual pair of vector spaces to a dual pair of a vector space and a cone.
\begin{definition}
A tuple $(X, Y, b)$ consisting of a vector space $X$, a cone $Y$ and a bilinear (with respect to nonnegative scalars in the second variable) mapping $b : X \times Y \to \mathbb{R}$ that is \textbf{nondegenerate} in the sense that for all $x \in X$ and all $y, z \in Y$,
\begin{equation}
b(x, \cdot) = 0 \implies x = 0
\quad\text{and}\quad
b( \cdot , y) = b( \cdot , z)  \implies y = z,
\end{equation}
is an \textbf{asymmetric dual pair}.
We shall use the abbreviation $\langle \cdot, \cdot \rangle$ instead of $b$ and also simply write that $X$ and $Y$ are \textbf{in duality}.
\end{definition}
Obviously, every ordinary dual pair is an asymmetric dual pair and an asymmetric dual pair is an ordinary dual pair precisely, when both spaces are vector spaces.
For the remainder of the section, we fix some $X$ and $Y$ in duality.
\begin{corollary}
$Y$ is cancellative.
\end{corollary}
\begin{proof}
Let $a, b, c \in Y$ and suppose that $a + c = b + c$.
Then, since $\langle \cdot, \cdot \rangle$ is bilinear,
\begin{equation}
\left< x, a \right> 
=
\left< x, a + c \right> - \left< x, c \right>
=
\left< x, b + c \right> - \left< x, c \right>
=
\left< x, b \right> \, ,
\end{equation}
for every $x \in X$.
Since $\langle \cdot, \cdot \rangle$ is nondegenerate, $a = b$ follows.
\end{proof}
Since $Y$ is cancellative, it can be embedded in a vector space.
Specifically, we set $\widetilde{Y} = \mathrm{span}\, \{ \langle \cdot, y \rangle : y \in Y \}$ and note that $X$ and $\widetilde{Y}$ are a dual pair of vector spaces in the obvious way.
Consequently, we may consider the \textbf{annihilator} of a subset $A \subset X$ defined as
\begin{equation}
\mathrm{ann}\, A
=
\left\{ y \in \widetilde{Y} \middle| \forall x \in A : \left< x, y \right> = 0 \right\} \, .
\end{equation}
We now proceed by introducing induced asymmetric locally convex topologies on $X$.
\begin{proposition}
\label{prop:WeakTop}
For every $y \in Y$, the map
\begin{equation}
p_y : X \to [0, \infty), x \mapsto \max \left\{ \left\langle x, y \right\rangle, 0 \right\}
\end{equation}
is an asymmetric seminorm.
Furthermore, the collection $\mathcal{P} = \{ p_y : y \in Y \}$ turns $(X, \mathcal{P})$ into a $T_0$ asymmetric locally convex space.
If $Y$ is a vector space, then $(X, \mathcal{P})$ is a Hausdorff, locally convex space.
There is also a natural partial order $\le$ on $X$, defined as
\begin{equation}
x \le x' \iff \forall y \in Y : \left< x, y \right> \le \left< x', y \right>
\end{equation}
for all $x, x' \in X$.
It coincides with the specialization order on $(X, \mathcal{P})$, i.e. $x \le x' \iff x' \in \mathrm{cl}\, \{ x \}$.
In particular, if $(X, \mathcal{P})$ is $T_1$, $x \le x'$ if and only if $x = x'$.
\end{proposition}
\begin{proof}
That $p_y$ is an asymmetric seminorm is obvious, so that $(X, \mathcal{P})$ is an asymmetric locally convex space.
Suppose $x \in X$ is such that for all $y \in Y$, $p_y(x) = 0$ and $p_y(-x) = 0$.
Then, $\langle x, y \rangle = 0$ for all $y \in Y$ and thus, $x = 0$ by the non-degeneracy of the dual pairing.
Consequently, $(X, \mathcal{P})$ is $T_0$.
If $Y$ is a vector space, it follows that $(X, \mathcal{P})$ is a $T_0$ locally convex space and thus Hausdorff.
It is easy to see that $\le$ is a partial order.
Moreover, for all $x, x' \in X$,
\begin{equation}
x' \in \mathrm{cl}\, \{ x \}
\iff
\forall \epsilon > 0 \forall y \in Y : p_y \left( x - x' \right) < \epsilon
\iff
\forall y \in Y : \left< x, y \right> \le \left< x', y \right> \, .
\end{equation}
\end{proof}
With $\mathcal{P}$ as in \zcref{prop:WeakTop}, we define the topology $\sigma(X, Y) = \tau_{\mathcal{P}}$ in accordance with the usual notation for dual pairs.
The conjugate topology then becomes $\overline{\sigma}(X, Y) = \tau_{\overline{\mathcal{P}}}$ and the smallest topology containing both of these is $\sigma_s(X, Y) = \tau_{\mathcal{P}_s}$.
It is instructive to verify that $\overline{\sigma}(X, Y) = \sigma(X, -Y)$ with $-Y$ considered within $\widetilde{Y}$ and that $\sigma_s(X, Y) = \sigma(X, \widetilde{Y})$.

There is of course the corresponding dual topology as well, which turns out to be much more robust, due to the vector space character of $X$.
\begin{proposition}[{\cite[Proposition 2.4.1]{src:Cobzaş:FunctionalAnalysisinAsymmetricNormedSpaces}}]
\label{prop:WeakStarTop}
\begin{equation}
q_x : Y \to [0, \infty), y \mapsto \max \left\{ \left\langle x, y \right\rangle, 0 \right\}
\end{equation}
is an asymmetric seminorm.
The topology $\sigma(Y, X)$ generated by $\mathcal{Q} = \{ q_x : x \in X \}$ is the topology of pointwise convergence, i.e. a net $(y_\alpha)_{\alpha \in I}$ in $Y$ $\sigma(Y, X)$-converges to $y \in Y$ if and only if
\begin{equation}
\forall x \in X : \lim_\alpha \left< x, y_\alpha \right> = \left< x, y \right> \, .
\end{equation}
\end{proposition}
It follows that this topology on $Y$ is ignorant of the cone character of $Y$.
It only sees the full vector space $\widetilde{Y}$, i.e. $\sigma(Y, X) = \sigma(\widetilde{Y}, X) \upharpoonright Y$ where $Y$ is considered as a subspace of $\widetilde{Y}$.
In other words, $\sigma(Y, X)$ is simply the weak-$\ast$ topology of the ordinary dual pairing $(X, \widetilde{Y})$ on the subspace $Y \subseteq \widetilde{Y}$, which justifies the notation.
A trivial consequence is that the annihilator of a subset $A \subseteq X$ is closed in $\sigma(\widetilde{Y}, X)$.

On a locally convex topological vector space $(X, \mathcal{P})$, a linear functional $\phi : X \to \mathbb{R}$ is continuous if and only if there are $p_1, \dots, p_n \in \mathcal{P}$ and some $C \ge 0$ such that $\vert \phi \vert \le C p_1 + \dots + C p_n$.
Similarly, one can show that for an asymmetric locally convex space $(X, \mathcal{P})$, a linear functional $\phi : X \to \mathbb{R}$ is upper semicontinuous if and only if there are $p_1, \dots, p_n \in \mathcal{P}$ and some $C \ge 0$ such that $\phi \le C p_1 + \dots + C p_n$ \cite[Proposition 2.1.11]{src:Cobzaş:FunctionalAnalysisinAsymmetricNormedSpaces}.
This enables the interpretation of $Y$ as the dual space of $X$ with respect to the pairing $(X, Y)$.
\begin{theorem}
$(X, \sigma(X, Y))^*$ is isomorphic to $Y$ in the sense that
\begin{enumerate}
\item for every $\phi \in (X, \sigma(X, Y))^*$, there is a unique $y \in Y$ such that $\phi = \langle \cdot, y \rangle$,
\item for every $y \in Y$, $\langle \cdot, y \rangle \in (X, \sigma(X, Y))^*$.
\end{enumerate}
\end{theorem}
\begin{proof}
$(ii)$:
Immediately clear from the above consideration.

$(i)$:
The existence follows from the much more general treatment of dual pairs of cones in \cite[p. 40, Theorem 3.8]{src:KeimelRoth:OrderedConesAndApproximation} while the uniqueness is an immediate consequence of the nondegeneracy.
\end{proof}
Consequently, we shall use the identification $Y \simeq (X, \sigma(X, Y))^*$ from here on.
With that in mind, it is clear that $\sigma(X, Y)$ coincides with the \textbf{weak topology} of the asymmetric locally convex space $(X, \sigma(X, Y))$ as defined in \cite[Section 2.4.6]{src:Cobzaş:FunctionalAnalysisinAsymmetricNormedSpaces}.
In particular, a net $(x_\alpha)_{\alpha \in I}$ $\sigma(X, Y)$-converges to $x \in X$ if and only if $\limsup_{\alpha} \langle x_\alpha, y \rangle \le \langle x, y \rangle$ for all $y \in Y$.
If $Y$ is a vector space, this is equivalent to $\limsup_{\alpha} \langle x_\alpha, y \rangle = \langle x, y \rangle$, resulting in the notion of the weak topology on locally convex topological vector spaces, justifying the notation $\sigma(X, Y)$.

The remainder of this section treats polar topologies with the goal of providing asymmetric analogoues of the bipolar theorem and the Mackey-Arens theorem.
Moreover, it is shown that every $T_0$ asymmetric locally convex space carries a polar topology, just as in the symmetric setting.
The author claims no novelty in these results, but they demonstrate the usefulness of asymmetric dual pairs.
Before getting started, let us mention two asymmetric version of the Hahn-Banch theorem.
\begin{theorem}[Hahn-Banach {\cite[Theorem 2.2.2]{src:Cobzaş:FunctionalAnalysisinAsymmetricNormedSpaces}}]
\label{thm:AsymmetricHahnBanachExtension}
Let $(Z, p)$ be an asymmetrically normed space and $z \in Z$.
There is some $\phi \in Z^*$ with $\phi(z) = p(z)$.
\end{theorem}
\begin{theorem}[Hahn-Banach {\cite[Theorem 2.2.9]{src:Cobzaş:FunctionalAnalysisinAsymmetricNormedSpaces}}]
\label{thm:AsymmetricHahnBanachSeparation}
Let $Z$ be an asymmetric locally convex space space $A, B \subseteq Z$ two nonempty, disjoint convex subsets such that $A$ is compact and $B$ is closed.
There is some $\phi \in Z^*$ with
\begin{equation}
\sup_{z \in A} \phi \left( z \right)
<
\inf_{z \in B} \psi \left( z \right) \, .
\end{equation}
\end{theorem}
\begin{corollary}
\label{cor:XandXDualAreInDuality}
Let $(Z, \mathcal{P})$ be a $T_0$ asymmetric locally convex space.
Then, $Z$ and $Z^*$ are in duality and $\sigma(Z, Z^*)$ is the weakest $T_0$ asymmetric locally convex topology on $Z$ preserving the dual space $Z^*$.
\end{corollary}
\begin{proof}
The only non-trivial part is to show that for all $z \in Z$,
\begin{equation}
\forall \phi \in Z^* : \phi \left( z \right) = 0 \implies z = 0 \, .
\end{equation}
So, suppose $z \in Z$ with $\phi(z) = 0$ for all $\phi \in Z^*$.
Since $Z$ is $T_0$, there is some asymmetric seminorm $p \in \mathcal{P}$ with $p(z) > 0$ or $p(-z) > 0$.
Accordingly, there are $\phi, \psi \in Z^*$ with $\phi(z) = p(z)$ and $\psi(-z) = p(-z)$.
Since one of these expressions is nonzero, we have a contradiction.
\end{proof}
$\sigma(Y, X)$ coincides with the \enquote{$w^\flat$}-topology defined in \cite[Section 2.4.6]{src:Cobzaş:FunctionalAnalysisinAsymmetricNormedSpaces} and by the previous considerations, we shall simply name it the \textbf{weak-$\ast$ topology}.
Call a set $A \subseteq X$ \textbf{$\sigma(X,Y)$-bounded} respectively $B \subseteq Y$ \textbf{$\sigma(Y,X)$-bounded} if
\begin{equation}
\forall y \in Y : \sup_{x \in A} \left< x, y \right> < \infty
\quad\text{respectively}\quad
\forall x \in X : \sup_{y \in B} \left< x, y \right> < \infty \, .
\end{equation}
It is easy to see that this boundedness coincides with the boundedness in ordinary dual pairs.
\begin{proposition}
Let $A \subseteq X$ and $B \subseteq Y$.
Then the function
\begin{equation}
p_B : X \to [0, \infty], x \mapsto \sup_{y \in B} p_y \left( x \right)
\end{equation}
is an asymmetric seminorm if and only if $B$ is $\sigma(Y, X)$-bounded.
Similarly,
\begin{equation}
q_A : Y \to [0, \infty], y \mapsto \sup_{x \in A} q_x \left( y \right)
\end{equation}
is an asymmetric seminorm if and only if $A$ is $\sigma(X, Y)$-bounded.
\end{proposition}
Consequently, we can consider \textbf{polar topologies} on $X$ respectively $Y$ generated by suitable sets on $Y$ respectively $X$, just as in the vector space setting.
Letting $\mathcal{G}$ denote an arbitrary collection of subsets of $Y$ respectively $X$, the corresponding topology induced by the asymmetric seminorms $\{ p_G : G \in \mathcal{G} \}$ respectively $\{ q_G : G \in \mathcal{G} \}$ will be called the \textbf{$\mathcal{G}$-topology} on $X$ respectively $Y$.
In particular $\sigma(X, Y)$ and $\sigma(Y, X)$ are - as usual - generated by the collection of all single point sets in $Y$ respectively in $X$.
\begin{proposition}
\label{prop:PolarProperties}
Let $A \subset X$ and $B \subseteq Y$.
Define the \textbf{polar} $A^\circ$ and the \textbf{prepolar} $\leftidx{^\circ}{B}{}$ as
\begin{equation}
A^\circ = \left\{ y \in Y : \sup_{x \in A} \left< x, y \right> \le 1 \right\}
\qquad\text{and}\qquad
\leftidx{^\circ}{B}{} = \left\{ x \in X : \sup_{y \in B} \left< x, y \right> \le 1 \right\} \, ,
\end{equation}
respectively.
Then $A^\circ$ is $\sigma(Y, X)$-closed and convex.
Moreover, it is \textbf{nonnegatively balanced} in the sense that $y \in A^\circ$ implies $t y \in A^\circ$ for all $t \in [0, 1]$.
Similarly, $\leftidx{^\circ}{B}{}$ is $\overline{\sigma}(X, Y)$-closed, convex and nonnegatively balanced.
\end{proposition}
\begin{proof}
That $A^\circ$ and $\leftidx{^\circ}{B}{}$ are convex and nonnegatively balanced is trivial.
Moreover, $A^\circ$ is $\sigma(Y, X)$-closed because the polar of $A$ in the pairing $(X, \widetilde{Y})$ is $\sigma(\widetilde{Y}, X)$-closed.
To see that $\leftidx{^\circ}{B}{}$ is $\overline{\sigma}(X, Y)$-closed, let $(x_\alpha)_{\alpha \in I}$ be a net in $\leftidx{^\circ}{B}{}$ that $\overline{\sigma}(X, Y)$-converges to some $x \in X$.
Then, for all $y \in Y$,
\begin{equation}
\left< x, y \right>
\le
\liminf_{\alpha} \left< x_\alpha, y \right>
\le
1 \, ,
\end{equation}
so that $x \in \leftidx{^\circ}{B}{}$.
\end{proof}

\begin{theorem}[Bipolar Theorem]
\label{thm:bipolar}
Let $A \subset X$ and $B \subseteq Y$.
Then,
\begin{align}
\label{eq:BipolarWeakConjugateClosureRep}
\leftidx{^\circ}{\left( A^\circ \right)}{}
=
\mathrm{cl}_{\overline{\sigma}(X,Y)} \mathrm{conv}\, \left[ A \cup \{ 0 \} \right]
\qquad\text{and}\qquad
\left( \leftidx{^\circ}{B}{} \right)^\circ
=
\mathrm{cl}_{\sigma(Y,X)} \mathrm{conv}\, \left[ B \cup \{ 0 \} \right] \, .
\end{align}
\end{theorem}
\begin{proof}
We begin, by showing the bipolar theorem for $A$.
Trivially, $0 \in \leftidx{^\circ}{( A^\circ )}{}$ and $\langle x, y \rangle \le 1$ for all $x \in A$, $y \in A^\circ$.
Hence, $A \subseteq \leftidx{^\circ}{( A^\circ )}{}$.
By \zcref{prop:PolarProperties}, it follows that
\begin{equation}
\mathrm{cl}_{\overline{\sigma}(X,Y)} \mathrm{conv}\, \left( A \cup \{ 0 \} \right)
\subseteq
\leftidx{^\circ}{\left( A^\circ \right)}{} \, .
\end{equation}
If strict inequality holds, there is some $x \in \leftidx{^\circ}{( A^\circ )}{}$ that is not contained in $\mathrm{cl}_{\overline{\sigma}(Y,X)} \mathrm{conv}\, ( A \cup \{ 0 \} )$.
Consequently, we may apply the Hahn-Banach theorem \cite[Theorem 2.2.9]{src:Cobzaş:FunctionalAnalysisinAsymmetricNormedSpaces} and find a $\phi \in (X, \overline{\sigma}(X, Y))^*$ with
\begin{equation}
\phi \left( x \right)
<
\inf \phi \left( \mathrm{cl}_{\overline{\sigma}(X,Y)} \mathrm{conv}\, ( A \cup \{ 0 \} ) \right) \, .
\end{equation}
Keeping in mind that $\overline{\sigma}(X, Y) = \sigma(X, -Y)$, it follows that there is a $y \in Y$ with $\langle \cdot, y \rangle = -\phi$ and thus,
\begin{equation}
\sup \left\{ \left< x', y \right>: x' \in \mathrm{cl}_{\overline{\sigma}(X,Y)} \mathrm{conv}\, ( A \cup \{ 0 \} ) \right\}
<
\left< x, y \right> \, .
\end{equation}
Since the left-hand side contains the origin, $\langle x, y \rangle > 0$ and by the strict inequality, there is some $r > 1$ with $r y / \langle x, y \rangle \in A^\circ$.
However, this contradicts $x \in \leftidx{^\circ}{( A^\circ )}{}$.
Consequently,
\begin{equation}
\leftidx{^\circ}{\left( A^\circ \right)}{}
=
\mathrm{cl}_{\overline{\sigma}(X,Y)} \mathrm{conv}\, \left( A \cup \{ 0 \} \right) \, .
\end{equation}

For $B$, the claim is obvious by recalling that $\sigma(Y, X)$ coincides with the subspace topology induced by $\sigma(\widetilde{Y}, X)$ and applying the bipolar theorem for dual pairs of vector spaces. 
\end{proof}
The final ingredient for the consideration of polar topologies is the Alaoglu theorem.
\begin{theorem}[Alaoglu {\cite[Theorem 2.4.30]{src:Cobzaş:FunctionalAnalysisinAsymmetricNormedSpaces}}]
Let $Z$ be an asymmetric locally convex space and $U \subseteq X$ a neighbourhood of the origin.
Then the polar of $U$ with respect to the pairing $(Z, Z^*)$ is weak-$\ast$-ly compact.
\end{theorem}
Let $\tau(X,Y)$ denote the \textbf{Mackey topology} on $X$, i.e. the $\mathcal{G}$-topology induced by
\begin{equation}
\mathcal{G}
=
\left\{ B \subseteq Y : B \text{ is convex and $\sigma(Y,X)$-compact} \right\} \, .
\end{equation}
Note that, in any Hausdorff, topological vector space, the convex, balanced hull of a convex, compact set is a compact disk \cite[p.66, 10.2]{src:SchaeferWolff:TVS}.
Hence, in the vector space setting, $\tau(X,Y)$ coincides with the ordinary Mackey topology.
The following \zcref[noref]{prop:MackeyTopIsFiner} follows immediately from the asymmetric Alaoglu theorem.
\begin{lemma}
\label{prop:MackeyTopIsFiner}
Let $Z$ be a vector space with a $T_0$ asymmetric locally convex topology $\alpha$.
Then, $\alpha \subseteq \tau(Z, (Z, \alpha)^*)$.
\end{lemma}
The proof of the following \zcref[noref]{thm:MackeyTopIsTheFinest} is identical to the one in \cite[p. 131, Theorem 3.2]{src:SchaeferWolff:TVS}.
\begin{theorem}
\label{thm:MackeyTopIsTheFinest}
$\tau(X, Y)^*$ is the finest $T_0$ asymmetric locally convex topology on $X$ with dual space $Y$.
\end{theorem}
Now the asymmetric analogues of the Mackey-Arens theorem follow.
\begin{corollary}
Let $Z$ be a $T_0$ asymmetric locally convex space with topology $\alpha$.
Then, $\sigma(Z, Z^*) \subseteq \alpha \subseteq \tau(Z, Z^*)$.
\end{corollary}
\begin{corollary}
Let $\alpha$ be a $T_0$ asymmetric locally convex topology on $X$ with $\sigma(X, Y) \subseteq \alpha \subseteq \tau(X, Y)$.
Then $(X, \alpha)^* = Y$.
\end{corollary}
The concept of equicontinuity easily generalises to the asymmetric setting as well.
\begin{definition}
Let $Z$ be an asymmetric locally convex space.
A subset $H \subseteq Z^*$ is \textbf{equicontinuous} if for each $\epsilon > 0$, there is a neighbourhood $U \subseteq Z$ of the origin, such that
\begin{equation}
\forall \phi \in H \forall x \in U : \phi \left( x \right) < \epsilon \, .
\end{equation}
\end{definition}
Moreover, every $T_0$ asymmetric locally convex topology is a $\mathcal{G}$-topology with the proofs being identical to those for locally convex topological vector spaces.
\begin{proposition}
Let $Z$ be an asymmetric locally convex space and $V \subseteq Z$ a neighbourhood of the origin.
Then the polar of $V$ with respect to the duality $(Z, Z^*)$ is equicontinuous.
\end{proposition}
\begin{theorem}
Let $Z$ be a $T_0$ asymmetric locally convex space and set
\begin{equation}
\mathcal{G} = \left\{
H \subseteq Z^* : \text{$H$ is equicontinuous}
\right\} \, .
\end{equation}
Then the $\mathcal{G}$-topology on $Z$ coincides with its original topology.
\end{theorem}
\section{Countability Conditions and Weak Compactness}
\label{sec:WeakCompactnessAndCountability}
By taking linear respectively conic combinations with rational coefficients, the following is obvious.
\begin{proposition}
\label{prop:SpanOfSeparableIsSeparable}
Let $X$ and $Y$ be in duality and equip both with some $\mathcal{G}$-topologies.
If $A \subseteq X$ respectively $B \subseteq Y$ is separable, then $\mathrm{span}\, A$ respectively $\mathrm{cone}\, B$ is separable.
\end{proposition}
For the remainder of this section, $X$ and $Y$ are taken to be in duality and are assumed to carry their respective weak topologies.
The following \zcref[noref]{prop:VectorDualitySeparableVersusCountablyCompact} appears to be folklore.
\begin{proposition}
\label{prop:VectorDualitySeparableVersusCountablyCompact}
Suppose $X$ is separable and $Y$ is a vector space.
Then every countably compact subset of $Y$ is compact and metrisable.
\end{proposition}
\begin{proof}
First, consider the functions $g_y : X \to \mathbb{R}, x \mapsto \langle x, y \rangle$ for all $y \in Y$.
By definition of $\sigma(X, Y)$, each $g_y$ is upper semicontinuous.
Moreover, $Y$ is a vector space, so each $g_{-y} = - g_y$ is upper-semicontinuous as well so that $g_y$ is, in fact, continuous.
Hence, for any dense sequence $(x_n)_{n \in \mathbb{N}}$ in $X$, the functions
\begin{equation}
F = \left\{ f_n : Y \to \mathbb{R}, y \mapsto \left< x_n, y \right> \right\}\, ,
\end{equation}
separate points in $Y$.
Letting $A \subseteq X$ be countably compact, it is immediate that $G = F \cup (-F)$ satisfies the conditions of \zcref{lem:USCSeparation}.
Since \zcref{prop:WeakStarTop} implies that $A$ is Hausdorff, the claim follows from \zcref{thm:CompactHausdorffCountableNetworkImpliesMetrisable}.
\end{proof}
In the asymmetric setting separability is not enough, because one cannot conclude that the family $F$ separates points in $A$.
However, a countable basis does suffice.
\begin{lemma}
\label{lem:CountableBasisImpliesMetrisableCompactaInDual}
If $X = \mathrm{span}\, E$ for some countable set $E \subseteq X$, every countably compact subset of $Y$ is compact and metrisable.
\end{lemma}
\begin{proof}
Let $A \subseteq Y$ be countably compact and $(x_n)_{n \in \mathbb{N}}$ an enumeration of $E$.
By the non-degeneracy, it is clear that whenever there are $y, z \in Y$ with $\langle x_n, y \rangle = \langle x_n, z \rangle$ for all $n \in \mathbb{N}$, then $y = z$.
Hence, the set $F = \{ f_n : n \in \mathbb{N} \}$ with
\begin{equation}
f_n : Y \to \mathbb{R}, y \mapsto \left< x_n, y \right> \, ,
\end{equation}
separates points in $A$ and the same proof as in \zcref{prop:VectorDualitySeparableVersusCountablyCompact} applies.
\end{proof}
In the reverse setting, separability works in conjunction with additional separation assumptions.
\begin{theorem}
\label{thm:SeparableDualT1AndCountablyCompactImpliesCompactWithCountableNetwork}
If $Y$ is separable, every subset $A \subseteq X$ that is $T_1$ and countably compact is compact and has a countable network.
\end{theorem}
\begin{proof}
As before, for every $y \in Y$, the function $g_y : X \to \mathbb{R}$ with $x \mapsto \langle x, y \rangle$ is upper semicontinuous.
Now, letting $x_1, x_2 \in A$ be distinct, the $T_1$ property implies that $x_1$ cannot be smaller or equal to than $x_2$ in terms of the parial order defined in \zcref{prop:WeakTop}.
Consequently, there is some $y' \in Y$ with $\langle x_1, y' \rangle < \langle x_2, y' \rangle$.
Letting $D \subseteq Y$ denote a countable dense set, we see from \zcref{prop:WeakStarTop}, that there is some $y \in D$ with $\langle x_1, y \rangle < \langle x_2, y \rangle$.
Consequently, \zcref{lem:USCSeparation} applies.
\end{proof}
\begin{corollary}
\label{cor:SeparableCompactnessCriterion}
If $Y$ is separable, a subset $A \subseteq X$ that is countably compact is compact if and only every totally ordered subset of $A$ with respect to the specialisation order has an upper bound.
\end{corollary}
\begin{proof}
This is just an application of \zcref{prop:T0CompactnessAbstractNonsense}.
\end{proof}
\begin{remark}
\label{rem:HausdorffSubsetAndSeparableDualImpliesMetrisability}
If the subset in question is Hausdorff, metrisability follows from \zcref{thm:CompactHausdorffCountableNetworkImpliesMetrisable}.
\end{remark}
The goal of this paper is to demonstrate the angelicity of certain subsets of $X$.
In order to achieve that, a better understanding of countably compact subsets is necessary.
\begin{proposition}
Let $K \subseteq Y$ be such that $\mathrm{cone}\, K = Y$.
Then, $X$ can be regarded as a subset of $\mathbb{R}^K$ via the identification $x \mapsto (\langle x, y \rangle)_{y \in K}$ which clearly also preserves the partial order $\le$.
\end{proposition}
For the remainder of this \zcref[noref]{sec:WeakCompactnessAndCountability}, let $K \subseteq Y$ denote such a set and equip $Z = [-\infty, \infty)^K$ with the usual product topology.

\begin{lemma}
\label{lem:ZClosureOfCountablyCompactSetIsCompact}
Let $A \subseteq X$ be countably compact.
Then, $\mathrm{cl}_Z A$ is compact.
\end{lemma}
\begin{proof}
The claim follows immediately, if $\mathrm{cl}_Z A$ is closed in $[-\infty, \infty]^K$, because the latter is compact by Tychonoff's theorem.
If $\mathrm{cl}_Z A$ is not closed in $[-\infty, \infty]^K$, there is some $y \in K$ as well as a net $(x_\alpha)_{\alpha \in I}$ in $A$ with $\lim_\alpha \langle x_\alpha, y \rangle = \infty$.
Consequently, there is a sequence $(x'_n)_{n \in \mathbb{N}}$ in $A$ with $\lim_\alpha \langle x'_n, y \rangle = \infty$.
However, $(x'_n)_{n \in \mathbb{N}}$ has a cluster point in $A$ which is a contradiction.
\end{proof}
Now, we can classify the countably compact sets that are compact.
\begin{theorem}
\label{thm:MaximalityCompactnessCondition}
A countably compact subset $A \subseteq X$ is compact if and only if
\begin{equation}
\label{eq:SymmetricClosureContainedInDownwardClosure}
\mathrm{Max}\, \mathrm{cl}_Z A \subseteq A ,
\end{equation}
where $\mathrm{Max}\, \mathrm{cl}_Z A$ is the set of all maximal elements of $\mathrm{cl}_Z A$.
\end{theorem}
\begin{proof}
$\Rightarrow$:
Assume that $A$ is compact and let $(x_\alpha)_{\alpha \in I}$ be a net in $A$ that $Z^K$-converges to some $x \in \mathrm{Max}\, \mathrm{cl}_Z A$.
By compactness, there is a cluster point $x' \in A$ and it is obvious that $x \le x'$.
However, since $x$ is maximal, $x = x'$ follows.

$\Leftarrow$:
Suppose that \zcref{eq:SymmetricClosureContainedInDownwardClosure} holds and let $(x_\alpha)_{\alpha \in I}$ be a net in $A$.
By \zcref{lem:ZClosureOfCountablyCompactSetIsCompact}, there is a $Z$-cluster point $x \in \mathrm{cl}_Z A$.
Using Zorn's lemma and the compactness of $\mathrm{cl}_Z A$, there is some $x' \in \mathrm{Max}\, \mathrm{cl}_Z A \subseteq A$ with $x \le x'$.
Trivially, $x'$ is a cluster point of $(x_\alpha)_{\alpha \in I}$.
\end{proof}
Finally, a simple criterion for countable tightness is identical to the vector space setting.
\begin{proposition}
\label{prop:DualIsSigmaCompactImpliesCountableTightness}
Suppose that $Y = \cup_{n \in \mathbb{N}} M_n$ for some sequence $(M_n)_{n \in \mathbb{N}}$ of compact subsets of $Y$.
Then $X$ has countable tightness.
\end{proposition}
\begin{proof}
The proof is identical to the one given in \cite[Theorem 3.54]{src:Fabian:BanachSpaceTheory}.
\end{proof}
\section{The space \texorpdfstring{$C_a(K)$}{Cₐ(K)}}
\label{sec:CaK}
There is extensive literature on the space of continuous functions on a given topological space with the topology of pointwise convergence (see e.g. \cite{src:Arkhangelskii1:TopologicalFunctionSpaces}).
In this section, some basic properties of an asymmetric version thereof are analysed.
Because this appears to be the first study of such spaces, no attempt is being made to extend the very general results known in the literature.
Instead, the aim is to uncover some easily accessible properties that do or do not carry over to the asymmetric setting.
Hopefully, this can provide some intuition and a basis for more general studies.

Let $K$ denote a compact Hausdorff space and $C(K)$ the linear space of real-valued continuous functions on $K$.
Moreover, for each $x \in K$, let
\begin{equation}
\delta_x : C(K) \to \mathbb{R}, f \mapsto f \left( x \right) \, .
\end{equation}
Set $\widehat{K} = \{ \delta_x : x \in K \}$, $L = \mathrm{cone}\, \widehat{K}$ and note that $C(K)$ and $L$ are in duality in the obvious way.
Let $C_a(K)$ denote the space $C(K)$ equipped with the $\sigma(C(K), L)$-topology (\enquote{a} for \enquote{asymmetric}).
This topology is similar to the one of $C_p(K)$, i.e. the topology of pointwise convergence.
\begin{corollary}
\label{cor:CaKConvergence}
A net $(f_\alpha)_{\alpha \in I}$ in $C_a(K)$ converges to $f \in C_a(K)$ if and only if for all $x \in K$,
\begin{equation}
\limsup_\alpha f_\alpha \left( x \right) \le f \left( x \right) \, .
\end{equation}
\end{corollary}
\begin{corollary}
$C_a(K)$ is hyperconnected and ultraconnected (see e.g. \cite[p. 29]{src:SteenSeebach:CounterexamplesInTopology}).
\end{corollary}
\begin{proof}
Letting $f, g \in C_a(K)$ be any two functions, it is clear that $\min\{ f, g \}$ lies in every neighbourhood of both $f$ and $g$.
Similarly, $\max\{ f, g \} \subseteq \overline{\{ f \}} \cap \overline{ \{ g \} }$.
\end{proof}
\begin{corollary}
$\sigma_s( C(K), L )$ coincides with the topology of $C_p(K)$.
\end{corollary}
\begin{corollary}
\label{cor:KToKHatHomeomorphism}
The mapping $K \to \widehat{K}: x \mapsto \delta_x$ is a homeomorphism when $\widehat{K}$ is equipped with the $\sigma(L, C(K))$-topology.
\end{corollary}
\begin{proof}
The weak-$\ast$ topology is ignorant of whether $L$ is a cone or a vector space and in the latter setting this is well-known (see e.g. \cite[Lemma 3.55]{src:Fabian:BanachSpaceTheory}).
\end{proof}
\begin{corollary}
\label{cor:SpecialisationOrderInCaKCoincidesWithFunctionOrder}
For $f, g \in C_a(K)$, $f \le g$ in the sense of \zcref{prop:WeakTop} if and only if $f \le g$ as functions.
\end{corollary}
An important property of the space $C_p(K)$ is its angelicity.
In particular, the class of (relatively) countably compact subsets coincides with the classes of (relatively) compact respectively (relatively) sequentially compact subsets.
A similar property holds for the classes of relative compactness in $C_a(K)$ albeit in a very trivial way.
\begin{lemma}
\label{lem:CaKRelativelyCountablyCompactSubnetBounded}
Let $A \subseteq C_a(K)$ be relatively countably compact.
Then, for every net $(f_\alpha)_{\alpha \in I}$ in $A$ there is a cofinal, directed subset $J \subseteq I$ such that
\begin{equation}
\sup_{x \in K} \limsup_{\beta \in J} f_\beta \left( x \right) < \infty \, .
\end{equation}
\end{lemma}
\begin{proof}
Suppose the claim is false and fix some $x_1 \in K$ with $\limsup_\alpha f_\alpha(x_1) > 1$.
Then $J_1 = \{ \alpha \in I : f_\alpha(x_1) \ge 1 \}$ is cofinal and directed and $(f_\beta)_{\beta \in J_1}$ a subnet with $f_\beta(x_1) \ge 1$ for every $\beta \in J_1$.
Because we assumed the claim to be false, we now have
\begin{equation}
\sup_{x \in K} \limsup_{\beta \in J_1} f_\beta \left( x \right) = \infty \, .
\end{equation}
Hence, proceeding by induction, there is a sequence $(x_n)_{n \in \mathbb{N}}$ in $K$ and for every $n \in \mathbb{N}$ a cofinal and directed set $J_{n+1} \subseteq J_n$ such that $f_\beta(x_m) \ge m$ for every $\beta \in J_n$ and $m \in \{ 1, \dots, n \}$.
Now, for every $n \in \mathbb{N}$, pick some $\alpha(n) \in J_n$ and consider the sequence $(f_{\alpha(n)})_{n \in \mathbb{N}}$.
By assumption, it has a cluster point $f \in C_a(K)$ and thus $f( x_m ) \ge m$ for every $m \in \mathbb{N}$.
Then, $\sup_{x \in K} f(x) = \infty$ contradicting the fact that the continuous function $f$ is bounded on the compact set $K$.
\end{proof}
\begin{corollary}
\label{cor:RelativeCompactnessEquivalences}
Let $A \subseteq C_a(K)$.
Then the following are equivalent:
\begin{enumerate}
\item $A$ is relatively compact
\item $A$ is relatively countably compact
\item $A$ is relatively sequentially compact
\end{enumerate}
\end{corollary}
\begin{proof}
$(i) \implies (ii)$ and $(iii) \implies (ii)$ are obvious.

$(ii) \implies (i)$:
Let $(f_\alpha)_{\alpha \in I}$ be a net in $A$.
By \zcref{lem:CaKRelativelyCountablyCompactSubnetBounded}, there is some $C \in \mathbb{R}$ and a subnet $(g_\beta)_{\beta \in J}$ with
\begin{equation}
\sup_{x \in K} \limsup_\beta g_\beta \left( x \right) \le C \, .
\end{equation}
Hence, $(g_\beta)_{\beta \in J}$ converges to the constant function with value $C$ on $K$.

$(ii) \implies (iii)$:
Let $(f_n)_{n \in \mathbb{N}}$ be a sequence in $A$.
By \zcref{lem:CaKRelativelyCountablyCompactSubnetBounded}, there is some $C \in \mathbb{R}$ and a subsequence $(g_n)_{n \in \mathbb{N}}$ with
\begin{equation}
\sup_{x \in K} \limsup_{n \to \infty} g_n \left( x \right) \le C \, .
\end{equation}
Hence, $(g_n)_{n \in \mathbb{N}}$ converges to the constant function with value $C$ on $K$.
\end{proof}
Consequently, $C_a(K)$ satisfies \zcref{def:angelicity:RelativeCompactnessEquivalence} of \zcref{def:angelicity}.
One may ask whether $C_a(K)$ is angelic as well, that is whether for every relatively countably compact subset $A \subseteq C_a(K)$ and every $\bar{a} \in \mathrm{cl}\, A$, there is a sequence in $A$ that converges to $\bar{a}$.
Unfortunately, in general, the answer turns out to be negative.
\begin{theorem}
\label{thm:Ca01IsNotAngelic}
$C_a([0,1])$ is not angelic.
\end{theorem}
\begin{proof}
Following \cite[Lemma II.3.5]{src:Arkhangelskii1:TopologicalFunctionSpaces}, there is a countable set $Z \subseteq C_p([0,1])$ such that for each $f \in Z$, $f([0,1]) \subseteq [0,1]$ and $0 \in \mathrm{cl}\, (Z \setminus \{ 0 \})$ but for which no sequence in $Z \setminus \{ 0 \}$ converges to $0$.
Since the topology of $C_a([0,1])$ is coarser than that of $C_p([0,1])$, $0$ is also in the $C_a(K)$-closure of $Z$.
Moreover, since all functions in $Z$ are nonnegative, a net in $Z$ $C_a([0,1])$-converges to $0$ precisely if it $C_p([0,1])$-converges to $0$.
At the same time, $Z$ is trivially seen to be relatively compact in $C_a(K)$, because every net in $Z$ $C_a(K)$-converges to the constant function with value one.
\end{proof}
However, the countable tightness does survive the passage from $C_p(K)$ to $C_a(K)$.
\begin{corollary}
The space $C_a(K)$ has countable tightness.
\end{corollary}
\begin{proof}
For each $m, n \in \mathbb{N}$, the set
\begin{equation}
L_{m,n} = \left\{ \sum_{i = 1}^m t_i \delta_{x_i} \middle| t_1, \dots, t_m \in [0, \infty), x_1, \dots x_m \in K : \sum_{i = 1}^m t_i \le n \right\}
\end{equation}
is $\sigma(L,C_a(K))$-compact.
Since $L = \cup_{m,n \in \mathbb{N}} L_{m,n}$, \zcref{prop:DualIsSigmaCompactImpliesCountableTightness} applies.
\end{proof}
The failure of $C_a([0,1])$ to be angelic was demonstrated in \zcref{thm:Ca01IsNotAngelic} by picking a specific relatively compact set.
In fact, there are even compact subsets of $C_a([0,1])$ that fail to be sequential.
\begin{example}
The set $Z$ from \zcref{thm:Ca01IsNotAngelic} together with the constant function with value one is compact and $T_0$ but not $T_1$.
It is however, not sequential, because no sequence in $(Z \cup \{ 1 \}) \setminus \{ 0 \}$ converges to $0$.
\end{example}
\begin{example}
\label{ex:CompactT1NonSequential}
To construct a compact $T_1$ set that is not sequential, let $(f_n)_{n \in \mathbb{N}}$ be an enumeration of $Z \setminus \{ 0 \}$.
Pick a continuous function $h : [0, 3] \to \mathbb{R}$ with $h([0,1]) = \{ 1 \}$, $h([1, 2]) = [0, 1]$ and $h((2, 3]) \subseteq (-\infty, 0)$.
Now, for every $n \in \mathbb{N}$ pick a continuous extension $g_n : [0, 3] \to \mathbb{R}$ of $f_n$ such that
\begin{itemize}
\item $g_n([1, 1 + 4^{-n}]) \subseteq [0, 1]$,
\item $g_n((1 + 4^{-n}, 1 + 2 \cdot 4^{-n})) \subseteq (1, \infty)$,
\item $g_n([1 + 2 \cdot 4^{-n}, 1 + 3 \cdot 4^{-n}]) \subseteq [0, 1]$,
\item $g_n((1 + 3 \cdot 4^{-n}, 1 + 4 \cdot 4^{-n})) \subseteq (-\infty, 0)$,
\item $g_n([1 + 4 \cdot 4^{-n}, 2]) = \{ 0 \}$
\item $g_n = h$ on $[2, 3]$
\end{itemize}
Then the set $D = \{ g_n : n \in \mathbb{N} \} \cup \{ h, 0 \} \subseteq C_a([0, 3])$ is countably compact, since $g_n$ converges to $h$ and thus compact, because it is countable.
Note, that $0 \in \mathrm{cl}\, \{ g_n : n \in \mathbb{N} \}$ because $g_n$ converges pointwise to zero on $(1, 2]$ and $g_n \le 0$ on $[2, 3]$.
By construction, $D$ is $T_1$ but not sequential, since no sequence in $D \setminus \{ 0 \}$ converges to $0$.
\end{example}
While perhaps not obvious on first sight, \zcref{ex:CompactT1NonSequential} is not Hausdorff.
This can be seen directly, by considering points in $(1, 2]$ that are close to $1$.
In fact, the existence of a compact Hausdorff space of countable tightness that is not sequential is independent of ZFC.
This amounts to the Moore-Mrówka problem (see e.g. \cite{src:Balogh:CompactHausdorffCountableTightness} and \cite[Classic Problem VI]{src:TopologyProceedings2:Problems}).
However, compact Hausdorff subspaces are  Fréchet-Urysohn as will be shown in \zcref{cor:HausdorffCountablyCompactImpliesSequentiallyCompact}.
Hence, there cannot exist a compact Hausdorff counterexample in this setting.

In angelic spaces, the classes of compact, countably compact and sequentially compact subsets coincide, a feature also lost in the asymmetric setting.
\begin{example}
\label{ex:Ca0omega1IsNotAngelic}
Letting $\omega_1$ denote the first uncountable ordinal, there is a countably compact subset $A \subseteq C_a([0, \omega_1])$ that is not compact.
\end{example}
\begin{proof}
Set $A = \{ f_\alpha : \alpha < \omega_1 \}$, where for each ordinal $\alpha < \omega_1$,
\begin{equation}
f_\alpha : [0, \omega_1] \to \mathbb{R},
\beta \mapsto \begin{cases}
1 & \text{if } \beta \le \alpha \\
0 & \text{otherwise.}
\end{cases}
\end{equation}
Since the assignment $\alpha \mapsto f_\alpha$ is order-preserving and every countable set of ordinals in $[0, \omega_1)$ has an upper bound less than $\omega_1$, it is obvious that every sequence in $A$ has a cluster point in $A$.
However, the net $(f_\alpha)_{\alpha < \omega_1}$ has no cluster point in $A$.
\end{proof}
It is well-known that every subset of $C_p(K)$ is a $D$-space \cite{src:BuzyakovaHereditaryDProperty}, a statement no longer true in the asymmetric setting.
\begin{corollary}
The set $A$ from \zcref{ex:Ca0omega1IsNotAngelic} is not a $D$-space.
\end{corollary}
It is instructive to examine \zcref{ex:Ca0omega1IsNotAngelic} in view of the classification given in \zcref{thm:MaximalityCompactnessCondition}.
Since $\mathrm{cone}\, K = L$, let $Z = [- \infty, \infty )^K$ with its product topology and note that $\mathrm{cl}_Z A = A \cup \{ g \}$ with
\begin{equation}
g : [0, \omega_1] \to \mathbb{R}, 
\beta \mapsto \begin{cases}
1 & \text{if } \beta < \omega_1 \\
0 & \text{if } \beta = \omega_1 \, .
\end{cases}
\end{equation}
Clearly, $g$ is maximal in $\mathrm{cl}_Z A$ but is not contained in $A$ since it is not even continuous.
\begin{proposition}
\label{prop:CountablyCompactIsCompactCharacterisation}
Let $A \subseteq C_a(K)$ be countably compact.
Then, the following are equivalent:
\begin{enumerate}
\item $A$ is compact,
\item $A$ is a $D$-space,
\item $\mathrm{Max}\, \mathrm{cl}_Z A \subseteq A$.
\end{enumerate}
\end{proposition}
It turns out that Hausdorffness lets us bridge the different compactness properties.
For a topological space $Q$, let $\underline{C}(Q) \subseteq \mathbb{R}^Q$ denote the set of real-valued lower semicontinuous functions on $Q$.
Clearly, $( \underline{C}(Q), \mathrm{cone}\, \widehat{Q} )$ is an asymmetric dual pair and we shall equip $\underline{C}(Q)$ with the corresponding weak topology.
Consequently, a net $(f_\alpha)_{\alpha \in I}$ in $\underline{C}(Q)$ converges to some $f \in \underline{C}(Q)$ if and only if $\limsup_\alpha f_\alpha(q) \le f(q)$ for all $q \in Q$.
Note, however, that unlike in \zcref{cor:KToKHatHomeomorphism}, the map $Q \to (\widehat{Q}, \sigma(\mathrm{cone}\, \widehat{Q}, \underline{C}(Q))$ is in general \textbf{not} continuous!
\begin{lemma}
\label{lem:AsymmetricHausdorffQuotient}
Let $B \subseteq C_a(K)$ be Hausdorff, $T \subseteq B$ and $\pi_2 : \widetilde{L} \to \widetilde{L} / \mathrm{ann}\, T$ the canonical quotient map where $\widetilde{L}$ is considered with the $\sigma(\widetilde{L}, C_a(K))$-topology and $\widetilde{L} / \mathrm{ann}\, T$ carries the quotient topology.
Set $Q = \pi_2(\widehat{K})$ and define the map
\begin{equation}
\pi_1 : B \to \underline{C}(Q), \text{ with } \left( \pi_1 f \right) \left( q \right)
=
\inf \left\{ f \left( k \right) \middle| k \in K : \delta_k \in \pi_2^{-1} \left( \left\{ q \right\} \right) \right\} \ ,
\end{equation}
for all $f \in B$ and $q \in Q$.
Then, $\pi_1$ is well-defined and
\begin{enumerate}
\item \label{itm:ProjectionOnCoreEquality} for all $f \in T$ and all $k \in K$, $( \pi_1 f ) ( \pi_2 \delta_k ) = f( k )$,
\item for all $f \in T$, $\pi_1 f$ is a continuous function from $Q$ to $\mathbb{R}$,
\item $\pi_1$ is continuous
\item $\pi_1 \upharpoonright \mathrm{cl}_B T$ is injective with a Hausdorff image,
\item \label{itm:CoreContinuity} a net $(f_\alpha)_{\alpha \in I}$ in $T$ converges to $f \in \mathrm{cl}_B T$ if and only if $(\pi_1 f_\alpha)_{\alpha \in I}$ converges to $\pi_1 f$,
\item if $B$ is regular, $\pi_1 \upharpoonright \mathrm{cl}_B T$ is a homeomorphism onto its image,
\item \label{itm:CoreSeparation} for every $f \in \mathrm{cl}_B T$ and every $g \in B \setminus \{ f \}$, there is some $q \in Q$ with $(\pi_1 g)(q) < (\pi_1 f)(q)$.
\end{enumerate}
\end{lemma}
\begin{proof}
First, note that $\mathrm{ann}\, T$ is $\sigma(\widetilde{L}, C_a(K))$-closed.
Hence, $\widetilde{L} / \mathrm{ann}\, T$ is Hausdorff and thus $Q$ is a compact, Hausdorff space.
By \zcref{cor:KToKHatHomeomorphism} and the definition of the quotient topology, the map
\begin{equation}
\psi : K \to Q, k \mapsto \pi_2 \left( \delta_k \right)
\end{equation}
is continuous and thus also closed and proper.
Using $\psi$ instead of $\pi_2$, one obtains
\begin{equation}
\label{eq:pi1fDefinitionViaPsi}
\left( \pi_1 f \right) \left( q \right)
=
\inf \left\{ f \left( k \right) \middle| k \in \psi^{-1} \left( \left\{ q \right\} \right) \right\} \ ,
\end{equation}
for all $f \in B$ and all $q \in Q$.
Now, let $r \in \mathbb{R}$ and set $C = (\pi_1 f)^{-1}( (- \infty, r] )$ as well as $D = f^{-1}( (-\infty, r] )$.
Since $\psi$ is proper and $f$ is continuous, the infimum in \zcref{eq:pi1fDefinitionViaPsi} is, in fact, a minimum.
Hence, $q \in C$ implies that there is some $k \in K$ with $\psi(k) = q$ and $f(k) \le r$, i.e. $C \subseteq \psi(D)$.
The converse inequality is obvious, so that $C = \psi(D)$.
Since $D$ is closed, so is $C$.
Hence, $\pi_1 f$ is lower semicontinuous and $\pi_1$ is well-defined.

$(i)$:
If $f \in T$ and $k, k' \in K$ with $\psi(k) = \psi(k')$, we have $\delta_k - \delta_{k'} \in \mathrm{ann}\, T$ and thus,
\begin{equation}
\label{eq:CoreEqualityOnEquivalentElements}
f \left( k \right)
=
\left< f, \delta_k \right>
=
\left< f, \delta_{k'} \right>
+
\left< f, \delta_k - \delta_{k'} \right>
=
f \left( k' \right) \, .
\end{equation}
$(ii)$:
Fix some $f \in T$, $r \in \mathbb{R}$ and let $U = (\pi_1 f)^{-1}( (-\infty, r) )$ as well as $V = f^{-1}( (-\infty, r ) )$.
Obviously, $\psi(V) = U$ so that $U$ is open if $\psi^{-1}( \psi(V) )$ is open.
However, by \zcref{eq:CoreEqualityOnEquivalentElements}, $V = \psi^{-1}( \psi(V) )$.
Since $f$ is continuous, $V$ is open and thus $U$ is open.
Consequently, $\pi_1 f$ is upper semicontinuous and thus continuous.

$(iii)$:
Let $(f_\alpha)_{\alpha \in I}$ be a net in $B$ that converges to some $f \in B$.
Then,
\begin{equation}
\begin{aligned}
\limsup_{\alpha} \left( \pi_1 f_\alpha \right) \left( q \right)
&=
\limsup_{\alpha} \inf_{k \in \psi^{-1} ( \{ q \} )} f_\alpha \left( k \right) \\
&\le
\inf_{k \in \psi^{-1} ( \{ q \} )} \limsup_{\alpha} f_\alpha \left( k \right)
\le
\inf_{k \in \psi^{-1} ( \{ q \} )} f \left( k \right)
=
\left( \pi_1 f \right) \left( q \right) \, ,
\end{aligned}
\end{equation}
for all $q \in Q$.

$(iv)$:
Now, let $f, g \in \mathrm{cl}_B T$ and set $a = \pi_1 f$ as well as $b = \pi_1 g$.
Since $\pi_1$ is continuous, $U \cap \pi_1(T) \neq \emptyset$ for all open $U \subseteq \underline{C}(Q)$ with $U \cap \pi_1(\mathrm{cl}_B T) \neq \emptyset$.
In particular, for all open neighbourhoods $U$ of $a$ and $V$ of $b$ with $U \cap V \neq \emptyset$, $U \cap V$ is open and $U \cap V \cap \pi_1(T) \neq \emptyset$.
Hence, if $a$ and $b$ have no disjoint neighbourhoods, there is a net $(f_\alpha)_{\alpha \in I}$ in $T$ such that $(\pi_1 f_\alpha)_{\alpha \in I}$ converges to both $a$ and $b$.
However, then
\begin{equation}
f \left( k \right)
\ge
a \left( \psi(k) \right)
\ge
\limsup_\alpha \left( \pi_1 f_\alpha \right) \left( \psi(k) \right)
=
\limsup_\alpha f_\alpha \left( k \right) \, ,
\end{equation}
for all $k \in K$.
Hence, $(f_\alpha)_{\alpha \in I}$ converges to $f$ and similarly, to $g$ as well.
Since $B$ is Hausdorff, it follows that $f = g$.
Consequently, $\pi_1$ is injective and has a Hausdorff image.

$(v)$:
The forward implication is simply the continuity of $\pi_1$.
For the converse, let $f \in \mathrm{cl}_B T$ and $(f_\alpha)_{\alpha \in I}$ be a net in $T$ such that $(\pi_1 f_\alpha)_{\alpha \in I}$ converges to $\pi_1 f$.
Then, for all $k \in K$,
\begin{equation}
\limsup_{\alpha} f_\alpha \left( k \right)
=
\limsup_{\alpha} \left( \pi_1 f_\alpha \right) \left( \psi \left( k \right) \right)
\le
\left( \pi_1 f \right) \left( \psi \left( k \right) \right)
\le
f \left( k \right) \, .
\end{equation}

$(vi)$:
Consider the function $g = \pi_1^{-1} \upharpoonright \pi_1(T)$.
Applying \cite[p. 81, Theorem 1]{src:Bourbaki:GeneralTopologyI}, $g$ has a unique continuous extension to $\pi_1(\mathrm{cl}_B T)$ which by definition coincides with $\pi_1^{-1} \upharpoonright \pi_1(\mathrm{cl}_B T)$.

$(vii)$:
Suppose that for some $g \in B \setminus \{ f \}$, we have $\pi_1 f \le \pi_1 g$.
Let $(f_\alpha)_{\alpha \in I}$ be a net in $T$ that converges to $f$.
Then, for all $k \in K$,
\begin{equation}
\limsup_{\alpha} f_\alpha \left( k \right)
=
\limsup_{\alpha} \left( \pi_1 f_\alpha \right) \left( \psi \left( k \right) \right)
\le
\left( \pi_1 f \right) \left( \psi \left( k \right) \right)
\le
\left( \pi_1 g \right) \left( \psi \left( k \right) \right)
\le
g \left( k \right) \, .
\end{equation}
Hence, $(f_\alpha)_{\alpha \in I}$ also converges to $g$ which contradicts the Hausdorffness of $B$.
\end{proof}
\begin{theorem}
Let $B \subseteq C_a(K)$ be countably compact and Hausdorff.
Then $B$ is Fréchet-Urysohn and the $B$-closure of any countable subset of $B$ is compact and metrisable.
\end{theorem}
\begin{proof}
Let $T \subseteq B$ be countable.
Applying \zcref{lem:AsymmetricHausdorffQuotient}, we have the canonical projection $\pi_2 : \widetilde{L} \to \widetilde{L} / \mathrm{ann}\, T$, $Q = \pi_2(\widehat{K})$ and the continuous, injective function $\pi_1 : \mathrm{cl}_B T \to \underline{C}(Q)$ with a Hausdorff image that we shall label $C$.

Clearly, $(\mathrm{span}\, T, \mathrm{cone}\, Q)$ is a duality and the topology on $Q$ is given by the subspace topology of $\sigma( \mathrm{cone}\, Q, \mathrm{span}\, T)$.
Since $T$ is countable, \zcref{lem:CountableBasisImpliesMetrisableCompactaInDual} applies.
Hence, $Q$ is metrisable and thus, separable.
Unfortunately, the map $Q \to \sigma(\mathrm{cone}\, \widehat{Q}, \underline{C}(Q) )$ is not continuous in general and we cannot conclude that $\widehat{Q}$ with its weak-$\ast$ topology is separable.
This is a major difference to the $C_p(K)$-setting, where it suffices to consider continuous functions on $Q$ rather than lower semicontinuous ones.
Consequently, using \zcref{thm:SeparableDualT1AndCountablyCompactImpliesCompactWithCountableNetwork} is out of reach.

However, $Q$ has a countable network $\mathcal{N}$ of closed sets by \zcref{thm:CompactHausdorffCountableNetworkImpliesMetrisable}.
For every $S \in \mathcal{N}$, define the upper semicontinuous function
\begin{equation}
\phi_S : C \to \mathbb{R}, f \mapsto \min_{q \in S} f \left( q \right) \, ,
\end{equation}
which is well-defined because $S$ is compact and $f$ lower semicontinuous.
Now, let $f, g \in C$ be distinct and note that, because $C$ is $T_1$, there is some $q \in Q$ with $f(q) < g(q)$.
Picking some $r \in (f(q), g(q))$, note that the set $V = g^{-1}((r, \infty))$ is an open neighbourhood of $q$, because $g$ is lower semicontinuous.
Consequently, there is some $S \in \mathcal{N}$ with $q \in S$ which implies
\begin{equation}
\phi_S \left( f \right)
\le
f \left( q \right)
<
r
<
\phi_S \left( g \right) \, .
\end{equation}
Hence, the countable family $\{ \phi_S : S \in \mathcal{N} \}$ satisfies the conditions of \zcref{lem:USCSeparation, thm:CompactHausdorffCountableNetworkImpliesMetrisable}, demonstrating that $C$ is compact and metrisable.

Now, $\pi_1$ is a continuous bijection of a countably compact space to a Fréchet-Urysohn space.
Consequently, \cite[Lemma III.4.27]{src:Arkhangelskii1:TopologicalFunctionSpaces} (where the proof implicitly uses Hausdorffness) shows that $\pi_1$ is a homeomorphism.
In particular, $\mathrm{cl}_B T$ is compact and metrisable and the Fréchet-Urysohn property of $B$ follows from its countable tightness.
\end{proof}
\begin{corollary}
\label{cor:HausdorffCountablyCompactImpliesSequentiallyCompact}
Every Hausdorff, countably compact subset of $C_a(K)$ is sequentially compact.
\end{corollary}
In fact, every Hausdorff, countably compact subset of $C_a(K)$ is also a D-space, which we will show by demonstrating its monotonic monolithicity in a manner inspired by \cite[Proposition 2.6]{src:Gruenhage:DSpaces}.
\begin{definition}
Let $B \subseteq C_a(K)$ and let $\mathcal{B}$ denote a countable basis of the topology of $\mathbb{R}$.
For a finite $F \subseteq B$, define the collection
\begin{equation}
\mathcal{C}^B_F
=
\left\{
\bigcap_{n = 1}^m f_n^{-1} \left( C_n \right)
\middle|
m \in \mathbb{N}, f_1 \dots, f_m \in B, C_1 \dots, C_m \in \mathcal{B}
\right\} \bigcup \{ \{ K \} \}
\end{equation}
of subsets of $K$.
Now, for $S \in \mathcal{C}^B_F$ and $r \in \mathbb{Q}$, define
\begin{equation}
\left[ S, r \right]
=
\left\{
f \in B
\middle|
\forall x \in S : f \left( x \right) \ge r
\right\}
\end{equation}
and let $\mathcal{G}^B_F$ denote the collection of all finite intersections of such sets.
Clearly, $\mathcal{G}^B_F$ is countable and $E \subseteq F$ implies $\mathcal{G}^B_E \subseteq \mathcal{G}^B_F$.
Finally, for an arbitrary $T \subseteq B$, define
\begin{equation}
\mathcal{G}^B \left( T \right) = \bigcup_{F \in [T]^{< \omega}} \mathcal{G}^B \left( F \right) \, ,
\end{equation}
where $[T]^{< \omega}$ is the collection of all finite subsets of $T$ in the notation used in \cite{src:Gruenhage:DSpaces}.
Because the particular choice of $\mathcal{B}$ does not matter, it is excluded from the notation.
\end{definition}
\begin{theorem}
A countably compact, Hausdorff subspace $B \subseteq C_a(K)$ is monotonically monolithic.
In particular, $B$ is a D-space and compact.
\end{theorem}
\begin{proof}
It is immediate that the assignement $A \mapsto \mathcal{G}^B(A)$ for all subsets $A \subseteq B$ satisfies \zcref{itm:MonotonicMonolithicity:CardinalityBound,itm:MonotonicMonolithicity:Monotonicity,itm:MonotonicMonolithicity:Continuity} of \zcref{def:MonotonicMonolithicity}.
It remains to prove that $\mathcal{G}^B(A)$ is an external network for all $f \in \mathrm{cl}_B A$.
To see that, let $U \subseteq B$ be an open neighbourhood of $f$.
Clearly, $B \setminus U$ is closed and countably compact.
Now, pick a countable set $T \subseteq A$ such that $f \in \mathrm{cl}_B T$ and apply the construction in \zcref{lem:AsymmetricHausdorffQuotient} with $Q = \pi_2(\widehat{K})$. From \zcref{itm:CoreSeparation}, it follows that for all $g \in B \setminus U$, there is some $r \in \mathbb{Q}$ and some $q \in Q$ with $(\pi_1 g)(q) < r < (\pi_1 f)(q)$.
Since $\pi_1 f$ is lower semicontinuous on $Q$, the set $V = (\pi_1 f)^{-1}((r, \infty))$ is an open neighbourhood of $q$.
By construction, $Q$ carries the initial topology induced by the functions $\pi_1(T)$.
Consequently, there is a finite subset $F \subseteq T$ such that
\begin{equation}
q \in \bigcap_{h \in F} \left( \pi_1 h \right)^{-1} \left( I_h \right) \subseteq V
\end{equation}
for some open subsets $(I_h)_{h \in F}$ of $\mathbb{R}$.
Consequently, applying \zcref{itm:ProjectionOnCoreEquality} of \zcref{lem:AsymmetricHausdorffQuotient} and identifying $K$ with $\widehat{K}$, we find
\begin{equation}
\pi_2^{-1} \left( \{ q \} \right)
\subseteq
\bigcap_{h \in F} h^{-1} \left( I_h \right)
\subseteq
\pi_2^{-1} \left( V \right) \, .
\end{equation}
In particular, this shows that there is some $S \in \mathcal{C}^B_F$ with $q \in \pi_2(S) \subseteq V$.
With $G_g = [S, r] \in \mathcal{G}^B(T)$, we then immediately obtain $f \in G_g$ and $g \notin G_g$.
Hence, $\cup_{g \in B \setminus U} (B \setminus G_g)$ is an open cover of $B \setminus U$ and because $\mathcal{G}^B(T)$ is countable, it has a finite subcover enumerated by some elements $g_1, \dots, g_n \in B \setminus U$.
Hence, $G = \cap_{m = 1}^n G_{g_m} \subseteq U$ with $f \in G$.
Moreover, $G \in \mathcal{G}^B(T)$ because the latter is closed under finite intersections.
\end{proof}
\begin{corollary}
The following are equivalent for every Hausdorff subset $A$ of $C_a(K)$:
\begin{enumerate}
\item $A$ is compact
\item $A$ is countably compact,
\item $A$ is sequentially compact.
\end{enumerate}
Moreover, in the affirmative case, $A$ is Fréchet-Urysohn and thus angelic.
\end{corollary}
For Hausdorff subspaces of $C_a(K)$, one thus has very similar properties as for $C_p(K)$.
Moreover, there is a simple topological embedding $\iota : C_p(K) \to C_a(K \times \{ -1, 1 \} )$ with $(\iota f)( k, 1 ) = f(k)$ and $(\iota f)( k, -1 ) = -f(k)$.
Because all subspaces of $C_p(K)$ are $D$-spaces, this raises the question whether all Hausdorff subspaces of $C_a(K)$ are $D$-spaces respectively angelic.
The author has not been able to find any counterexamples.
However, the topology of Hausdorff subspaces of $C_a(K)$ can be very intricate.
In particular, there are subspaces that are homeomorphic to the Sorgenfrey line.
\begin{example}[ChatGPT-6]
\label{ex:SorgenfreyEmbedding}
Let $D$ denote the double arrow space $[0, 1] \times \{0, 1\}$ with its lexicographic order topology and set $K = D \sqcup \{ \ast \}$ by adjoining the isolated point $\{ \ast \}$.
Pick any decreasing homeomorphism $\phi : \mathbb{R} \to (0, 1)$ and for each $t \in R$, define $f_t(\ast) = -\phi(t)$ and
\begin{equation}
f_t \left( d \right)
=
\begin{cases}
1 & \text{if } d \prec \left( \phi(t), 1 \right) \\
0 & \text{if } d \succeq \left( \phi(t), 1 \right)
\end{cases}
\end{equation}
for all $d \in D$, where $\preceq$ denotes the lexicographic order on $D$.
Because $(\phi(t), 1)$ is immediately preceeded by $(\phi(t), 0)$, it follows that $f_t$ is continuous.
Now, set $B = \{ f_t : t \in \mathbb{R} \}$ and consider the bijection $J : \mathbb{R}_l \to B$ with $t \mapsto f_t$ where $\mathbb{R}_l$ denotes the Sorgenfrey line $\mathbb{R}$ with the topology generated by half-open intervals of the form $[a, b)$.
For any $r, s \in \mathbb{R}$, note that
\begin{equation}
J^{-1} \left( \left\{ f \in B : f \left( \phi(r), 1 \right) < \frac{1}{2} \right\} \right)
=
\left[ r, \infty \right)
\end{equation}
and
\begin{equation}
J^{-1} \left( \left\{ f \in B : f \left( \ast \right) < -\phi \left(s \right) \right\} \right)
=
\left( -\infty, s \right) \, .
\end{equation}
Because $J$ is a bijection, it follows that $J^{-1}$ is continuous.
At the same time, we also have
\begin{equation}
J^{-1} \left( \left\{ f \in B : f \left( \phi(r), 0 \right) < \frac{1}{2} \right\} \right)
=
\left[ r, \infty \right)
\end{equation}
and combining all three preceeding equations, we see that $J$ is continuous.
\end{example}
It should be noted that the Sorgenfrey line cannot be embedded in $C_p(K)$.
To see that, note that $\mathbb{R}_l$ is separable and $C_p(K)$ has the property, that for any subset, the corresponding relative closure of a separable space has a countable network \cite[Chapter 6]{src:Arkhangelskii1:TopologicalFunctionSpaces}.
However, $\mathbb{R}_l$ does not have a countable network.
While the Sorgenfrey line itself may appear as \enquote{not too exotic}, one can also embed countable powers of subsets of $C_a(K)$ into a $C_a$-space.
\begin{proposition}
Define $Q = (\sqcup_{n \in \mathbb{N}} K) \cup \{ \infty \}$ to be the one-point compactification of the union of countably many disjoint copies of $K$.
Moreover, for some subset $B \subseteq C_a(K)$, define
\begin{equation}
J : B^{\mathbb{N}} \to C_a(Q)
\qquad
\text{with }
J \left( F \right) \left( \left\{ m, k \right\} \right)
=
2^{-m} \frac{F_m \left( k \right)}{1 + \left| F_m \left( k \right) \right|}
\end{equation}
for any $F \in B^{\mathbb{N}}$, $m \in \mathbb{N}$ and $k \in K$.
Moreover, set $J(F)(\infty) = 0$.
Then, $J$ is a topological embedding.
\end{proposition}
Consequently, even the set $\mathbb{R}_l^{\mathbb{N}}$ can be realised as a subset of a $C_a$-space and note that the latter has been shown to be a $D$-space in \cite{src:Patrakeev:CountableSorgenfreyPower}.

As a final note, the utility of Hausdorffness seems to be very powerful for subsets of $C_a(K)$.
However, it is also exceedingly restrictive and especially in light of \zcref{cor:SeparableCompactnessCriterion,thm:MaximalityCompactnessCondition}, one would rather want statements for $T_1$ spaces.
\section{The Weak Topology on an Asymmetrically Normed Space}
\label{sec:WeakTopologyOnAsymmetricallyNormedSpace}
It is well-known that the weak topology on a normed space is angelic (see e.g \cite[Proposition 3.108]{src:Fabian:BanachSpaceTheory}).
In this section we apply the results from the preceeding \zcref[noref]{sec:CaK} to the setting of asymmetrically normed spaces.
\begin{proposition}
Let $(X, p)$ be an asymmetrically normed space.
Then, $(X, p)$ with its weak topology is homeomorphic to a subset of $C_a(K)$ where
\begin{equation}
K = \left\{ \phi \in Y : \sup_{p(x) \le 1} \phi \left( x \right) \le 1 \right\}
\end{equation}
is compact in the $\sigma(Y, X)$ topology.
\end{proposition}
\begin{proof}
That $K$ is compact is well-known \cite[Theorem 2.4.3]{src:Cobzaş:FunctionalAnalysisinAsymmetricNormedSpaces}.
Let $f : (X, (X, p)^*) \to C_a(K)$ be given by $f(x)(\phi) = \phi(x)$ for all $x \in X$ and $\phi \in K$.
Then, $f$ is a homeomorphism by \zcref{prop:WeakStarTop,cor:XandXDualAreInDuality,cor:CaKConvergence}.
\end{proof}
\begin{corollary}
An asymmetrically normed space has countable tightness in its weak topology.
\end{corollary}
\begin{corollary}[Eberlein-Šmulian]
Let $X$ be an asymmetrically normed space and $A \subseteq X$ a subset equipped with the weak topology.
If $A$ is Hausdorff, the following are equivalent:
\begin{itemize}
\item $A$ is sequentially compact,
\item $A$ is countably compact,
\item $A$ is compact.
\end{itemize}
\end{corollary}
\section*{Declarations}
\subsection*{Acknowledgements}
I wish to thank an anonymous referee for helpful advice on improving this paper.
\subsection*{Data Availability}
No datasets were generated or analysed for this study.
\subsection*{Conflict of Interests}
I declare that I have no conflict of interests.
\subsection*{AI Disclosure}
I have used ChatGPT-6 for brainstorming and for producing \zcref{ex:SorgenfreyEmbedding}.
\appendix
\printbibliography
\end{document}